\documentclass[reqno,12pt,letterpaper]{amsart}
\usepackage{amsmath,amssymb,amsthm,graphicx,mathrsfs,url,bbm,array,enumerate}
\usepackage[justification=centering]{caption}
\usepackage[usenames,dvipsnames]{xcolor}
\usepackage[colorlinks=true,linkcolor=Red,citecolor=Green]{hyperref}
\usepackage{tikz-cd}
\usepackage[
    backend=biber,
    style=alphabetic,
    giveninits=true
]{biblatex}
\DeclareFieldFormat{pages}{#1}
\renewbibmacro{in:}{%
  \ifentrytype{article}
    {}
    {\bibstring{in}%
     \printunit{\intitlepunct}}}
\DeclareFieldFormat
  [article,inbook,incollection,inproceedings,patent,thesis,unpublished]
  {title}{\mkbibemph{#1}}
\DeclareFieldFormat{journaltitle}{#1\isdot}
\DeclareFieldFormat[article]{volume}{\mkbibbold{#1}}
\DeclareFieldFormat[article]{number}{\bibstring{number}\addnbspace #1}

\renewbibmacro*{journal+issuetitle}{%
  \usebibmacro{journal}%
  \setunit*{\addspace}%
  \iffieldundef{series}
    {}
    {\newunit
     \printfield{series}%
     \setunit{\addspace}}%
  \printfield{volume}%
  \setunit{\addspace}%
  \usebibmacro{issue+date}%
  \setunit{\addcomma\space}%
  \printfield{number}%
  \setunit{\addcolon\space}%
  \usebibmacro{issue}%
  \setunit{\addcomma\space}%
  \printfield{eid}
  \newunit}

\newcolumntype{L}{>{$}l<{$}}

\def\?[#1]{\textbf{[#1]}\marginpar{\Large{\textbf{??}}}}
\newtheorem{thm}{Theorem}
\newtheorem{prop}{Proposition}

\newtheorem{lem}[prop]{Lemma}
\newtheorem{cor}[prop]{Corollary}
\newtheorem{rem}[prop]{Remark}

\numberwithin{equation}{section}
\numberwithin{prop}{section}

\theoremstyle{definition}

\renewcommand{\Re}{\mathop{\rm Re}\nolimits}

\DeclareMathOperator{\Vol}{Vol}

\newcommand{\Ocal}{{\mathcal O}}

\newcommand{\CC}{{\mathbb C}}

\newcommand{\id}{{\rm id}}

\newcommand{\R}{\mathbb{R}}
\newcommand{\C}{\mathbb{C}}
\newcommand{\Lg}{\mathfrak{g}}
\newcommand{\Lk}{\mathfrak{k}}
\newcommand{\Lp}{\mathfrak{p}}
\newcommand{\Lh}{\mathfrak{h}}
\newcommand{\Ls}{\mathfrak{s}}

\newcommand{\Diff}{\mathrm{Diff}}
\newcommand{\Sym}{\mathrm{Sym}^2}

\begin{document}
\title[Spectral asymptotics for kinetic Brownian motion]{Spectral asymptotics for kinetic Brownian motion on locally symmetric spaces}

\author{Qiuyu Ren}
\email{qiuyu\_ren@berkeley.edu}
\address{Department of Mathematics, Evans Hall, University of California,
Berkeley, CA 94720, USA}

\author{Zhongkai Tao}
\email{ztao@math.berkeley.edu}
\address{Department of Mathematics, Evans Hall, University of California,
Berkeley, CA 94720, USA}

\begin{abstract}
We prove the strong convergence of the spectrum of the kinetic Brownian motion to the spectrum of base Laplacian for a large class of compact locally Riemannian homogeneous spaces, in particular all compact locally symmetric spaces. This generalizes recent work of Kolb--Weich--Wolf~\cite{Kolb2022} on constant curvature surfaces.
\end{abstract}

\maketitle
\section{Introduction}
Kinetic Brownian motion is a stochastic process that interpolates between the geodesic flow and Brownian motion. It was studied by several authors ~\cite{FrLe07}~\cite{GS13}~\cite{ABT15}~\cite{Li16} recently from the stochastic point of view. In this paper we focus on the spectral theory and prove the convergence of spectrum of the kinetic Brownian motion to the base Laplacian for a family of locally Riemannian homogeneous spaces. For additional references and background, see ~\cite{Kolb2022}.

More precisely, let $(M,g)$ be a closed Riemannian manifold of dimension $n\geq 2$, and let $X$ be the geodesic vector field on the unit tangent bundle $SM$. Let $\Delta_M$ be the Laplacian on $M$, and $\Delta_V$ be the vertical Laplacian on $SM$, namely $(\Delta_Vf)|_{S_pM}=\Delta_{S_pM}(f|_{S_pM})$ for every $p\in M$ (Our sign convention is that Laplacians have nonnegative eigenvalues). For $\gamma\in(0,\infty)$, define an operator \begin{align}P_\gamma=-\gamma X+c_n\gamma^2\Delta_{V}\end{align} on $SM$, where $c_n=1/n(n-1)$. 

We consider the spectrum of the operator $P_\gamma:D(P_\gamma)=\{u\in L^2(SM): P_\gamma u\in L^2\}\to L^2(SM)$, which we denote by $\sigma(P_\gamma)$. By hypoellipticity, $P_\gamma$ has discrete spectrum with finite multiplicity (see e.g. ~\cite[Proposition 2.1]{Kolb2022}).
We prove the following Theorem.

\begin{thm}\label{thm:Conv}
Suppose $M$ is a closed locally symmetric space. On any bounded open set $U\Subset \mathbb{C}$, we have
\begin{align}\label{spec_conv}
    \sigma(P_\gamma)\cap U\to \sigma(\Delta_M)\cap U, \quad \gamma\to\infty
\end{align}
with multiplicities. Moreover, in $B(L^2(SM))$ topology, we have
\begin{align}\label{res_conv}
    (P_\gamma-\lambda)^{-1}\to (\Delta_M-\lambda)^{-1},\quad \gamma\to \infty
\end{align}
uniformly for $\lambda\in U\Subset\mathbb{C}\setminus \sigma(\Delta_M)$. 
\end{thm}
In fact, later we will prove the theorem for a more general class of Riemannian manifolds (see Theorem~\ref{thm:conv_refined}).

Here, the resolvant $(\Delta_M-\lambda)^{-1}$ in the statement is a priori only defined on $L^2(M)$, but $L^2(M)\subset L^2(SM)$ naturally by pullback and we extend $(\Delta_M-\lambda)^{-1}$ to an operator on $L^2(SM)$ by assigning zero on $L^2(M)^\perp$. This will be explained again in Section~\ref{sec:decomp}.


Drouot~\cite{Dr17} and Smith~\cite{Sm20} studied the limit of $P_\gamma$ when $\gamma\to 0$ for Anosov geodesic flows using semiclassical analysis. The operator $P_\gamma$ is an analogy to the hypoelliptic Laplacian $L_\gamma$ introduced by Bismut~\cite{bismut2005hypoelliptic}. The main analytic difference is that $P_\gamma$ is defined on the unit tangent bundle $SM$ (or the unit cotangent bundle if we want) while $L_\gamma$ is defined on the cotangent bundle $T^*M$. Bismut--Lebeau~\cite{BL08} proved a statement similar to Theorem~\ref{thm:Conv} for $L_\gamma$ in the limit $\gamma\to \infty$ for arbitrary closed manifolds $M$. Bismut~\cite{bismut2011hypoelliptic} also studied the limit of $L_\gamma$\footnote{In fact, the hypoelliptic Laplacian considered here is slightly different than the previous one.} when $\gamma\to 0$ for locally symmetric spaces $M$ of noncompact type and obtained formulas for orbit integrals. This was used by Shen~\cite{shen2017analytic} to give a proof of Fried's conjecture for locally symmetric spaces. See Shen ~\cite{shen2021analytic} for a survey on Fried's conjecture. As noticed in Kolb--Weich--Wolf~\cite{Kolb2022}, it seems possible to adapt the method of ~\cite{BL08} to $P_\gamma$ to study the $\gamma\to\infty$ limit. In this paper, however, instead of attacking the general case following this line, we give a simple proof in the locally symmetric case by exploiting the symmetry following the idea of ~\cite{KWW2019}~\cite{Kolb2022}.
\begin{table}[h]
\begin{tabular}{|c|c|c|}\hline
 &$\gamma\to0$ &$\gamma\to\infty$ \\\hline
$P_\gamma$ &negatively curved &(slightly more general than) locally symmetric \\\hline
$L_\gamma$ &locally symmetric &no condition \\\hline
\end{tabular}
\caption{Condition on $M$ for which we have a control of the limit}
\end{table}

Kolb--Weich--Wolf~\cite{Kolb2022} showed that the spectrum of $P_\gamma$ converges to the spectrum of the base Laplacian $\Delta_M$ in a weak sense, when $M$ is a constant curvature compact surface. While we follow the same idea of  decomposing $L^2(SM)$ into Casimir eigenspaces, our work is a generalization of theirs in two aspects. First, we extend the statement to one where $M$ is in a large family of Riemannian homogeneous spaces including closed locally symmetric spaces. This is achieved by a general construction of a Casimir operator in Proposition \ref{prop:Cas}. Second, we improve the weak convergence of spectrum to the convergence of spectrum in locally compact topology, and we also prove the resolvent convergence. In particular, we use a Grushin problem instead of Kato's perturbation theory, which gives us more precise information about the spectrum and resolvant.

The structure of the paper is as follows. In Section~\ref{sec2} we exploit the symmetry on $M$ to construct an operator $\Omega$ on $SM$, called a Casimir operator, which commutes with all operators on $SM$ we will be concerned with. In Section~\ref{sec3}, we prove Theorem~\ref{thm:Conv} using a Grushin problem on each eigenspace of the Casimir operator $\Omega$. The stronger convergence is proved by showing that there are only finitely many eigenspaces of the Casimir operator that would matter. This is proved by a contradiction argument, with a careful analysis of projections to spherical harmonics.

\subsection*{Acknowledgement}
We would like to thank Alexis Drouot and Maciej Zworski for motivating this project and many helpful discussions. ZT was partially supported by
National Science Foundation under the grant DMS-1901462 and by Simons
Targeted Grant Award No. 896630.

\section{A Casimir operator on \texorpdfstring{$SM$}{SM}}\label{sec2}
The goal of this section is to construct a certain second order differential operator $\Omega$ on the unit tangent bundle $SM$ of $M$ where $M$ is Riemmannian homogeneous with some extra properties. The exact statement is Proposition~\ref{prop:Cas}. All constructions and proofs relevant to this paper will be given in Section~\ref{sec2.1}. Section~\ref{sec2.2} exploits the construction in more detail in the special case when $M$ is locally isotropic by giving an alternative more explicit construction.

\subsection{The construction}\label{sec2.1}
We first recall and fix the use of some terminologies. A \textit{Riemannian homogeneous space} is a Riemannian manifold $N$ together with the action of a Lie group $G$ that is effective, transitive, and isometric. We usually write $N=G/K$ where $K$ is the isotopy group of a chosen point on $N$. A \textit{(complete) locally Riemannian homogeneous space} is a smooth quotient $M=\Gamma\backslash N$ where $N=G/K$ is a Riemannian homogeneous space, and $\Gamma<G$ is a discrete subgroup acting freely and properly discontinuously on $N$.

An importance class of Riemannian homogeneous spaces is given by symmetric spaces. A \textit{symmetric space} is a Riemannian manifold $N$ such that for any $p\in N$, there exists an isometry of $N$ that fixes $p$ and induces the negation map on $T_pN$. A symmetric space is Riemannian homogeneous, thus can be written $N=G/K$, where $G$ is the isometry group of $N$. A \textit{locally symmetric space} is a Riemannian manifold $M$ that is locally modelled on some symmetric space $N=G/K$; if it is complete, then $M=\Gamma\backslash G/K$ for some $\Gamma<G$. References on Riemannian homogeneous spaces and symmetric spaces include ~\cite{arvanitogeorgos2003introduction}~\cite{helgason1979differential}~\cite{wolf2011spaces}.
 
As usual, we use script lowercase letters to denote the Lie algebra of a Lie group. For example $\Lg$ denotes the Lie algebra of $G$.
\begin{prop}\label{prop:Cas}
Let $M=\Gamma\backslash G/K$ be a locally Riemannian homogeneous space. Suppose there exists a $\Gamma$-invariant element in $\Sym\Lg$ whose image in $\Sym(\Lg/Ad_g\Lk)$ is negative definite for every $g\in G$, then there is a differential operator $\Omega$ on $SM$ satisfying the following:
\begin{enumerate}[(1)]
\item $\Omega$ is a second order symmetric operator, and its symbol in the horizontal part is negative definite.
\item $\Omega$ commutes with the vertical Laplacian $\Delta_V$, the total Laplacian $\Delta$, and the geodesic vector field $X$ on $SM$.
\end{enumerate}
\end{prop}
In the sequel, the operator $\Omega$ so constructed will be called a \textit{Casimir operator} on $SM$. Here, our convention for symbols of differential operators is the one used in differential geometry. Namely, at a local level, the symbol of the $k$-th order operator $P=\sum_{|\alpha|\le k}c_\alpha\partial^\alpha$ is $\sum_{|\alpha|=k}c_\alpha\xi^\alpha$. When $k=2$, this differs from the microlocal analysis convention by a sign.

The condition on $M$ in Proposition~\ref{prop:Cas} is not too restrictive. For example, it holds for all finite $\Gamma$ (thus all compact $G$), as well as all (complete) locally symmetric $M$ as we will show in a moment.
\begin{rem}\label{rem:(G,N)}
Assuming the element in $\Sym\Lg$ in Proposition~\ref{prop:Cas} is actually $G$-invariant, our argument will show the same for the slightly wider class of $(G,N)$-manifolds $M$ (i.e. manifolds that is locally modelled on the $G$-manifold $N=G/K$, with transition functions in $G$). Since for our analytic application we only care about compact (in particular complete) $M$, in which case all $(G,N)$-manifolds are of the form $\Gamma\backslash G/K$, we don't bother to define the notion of $(G,N)$-manifolds.
\end{rem}
\begin{cor}\label{cor:loc_sym}
If $M$ is a locally symmetric space (in particular, if it has constant curvature), then there exists an operator $\Omega$ on $SM$ satisfying the conditions above.
\end{cor}
In particular, there exists a Casimir operator $\Omega$ as above when $M$ has constant curvature. In fact, an explicit choice of $\Omega$ in this special case is given in Corollary~\ref{cor:const_curv}.
\begin{proof}
We only give the proof when $M$ is complete. The general case (which we won't need) follows from Remark~\ref{rem:(G,N)}. 

Let $N$ be the universal cover of $M$, which is a symmetric space. It splits into $N=N_0\times N_1\times N_2$, where $N_0=\R^r$ is euclidean, $N_1$ is of compact type and $N_2$ is of noncompact type. The isometry group of $N$ is $G=G_0\times G_1\times G_2$ where $G_i$ is the isometry group of $N_i$ (see e.g. Wolf~\cite[Theorem 8.3.8,8.3.9]{wolf2011spaces} for these statements). We claim there is a $G$-invariant element in $\Sym\Lg$ whose image is negative definite in each $\Sym(\Lg/Ad_g\Lk)$, which is sufficient for applying Proposition~\ref{prop:Cas}. It suffices to check the claim for each factor $N_i=G_i/K_i$.

When $i=1$, the Killing form of $\Lg_1$ is negative definite, thus its dual is a $G_1$-invariant negative definite element in $\Sym\Lg_1$.

When $i=2$, the Killing form of $\Lg_2$ is nondegenerate on $\Lg_2$, negative definite on each $Ad_g\Lk_2$ and positive definite on the orthogonal complement of each $Ad_g\Lk_2$. Therefore its negative dual is a $G$-invariant element in $\Sym\Lg_2$ whose image in each $\Sym(\Lg_2/Ad_g\Lk_2)$ is negative definite.

When $i=0$, $G_0$ is the euclidean group $E(r)=\left\{\left(\begin{smallmatrix}g&v\\0&1\end{smallmatrix}\right)\colon g\in O(r),\ v\in\R^r\right\}$ and $K_0$ is the upper-left $O(r)$. The Lie algebras are $\Lg_0=\left\{\left(\begin{smallmatrix}X&v\\0&0\end{smallmatrix}\right)\colon X\in\mathfrak{o}(r),\ v\in\R^r\right\}$ and $\Lk_0=\mathfrak{o}(r)$. It is straightforward to check the element $-\sum_{i=1}^n\left(\begin{smallmatrix}0&e_i\\0&0\end{smallmatrix}\right)^2\in\Sym\Lg_0$ ($e_i$ is the $i$-th coordinate vector) satisfies the desired conditions.
\end{proof}
\begin{proof}[Proof of Theorem~\ref{prop:Cas}]
If $\mathcal{M}$ is any manifold with a left $G$-action, the negative associated vector field construction gives a Lie algebra homomorphism $\Lg\to\mathfrak{X}(\mathcal{M})$, which induces a $G$-equivariant filtered-algebra homomorphism $U(\Lg)\to\Diff^*(\mathcal{M})$ where $U(\Lg)$ is the universal envelopping algebra of $\Lg$, and $\Diff^*(\mathcal{M})$ is the algebra of $C^\infty$-differential operators on $\mathcal{M}$. The induced map on the associated graded algebra is $\mathrm{Sym}(\Lg)\to\Gamma(\mathcal{M};\mathrm{Sym}(T\mathcal{M}))$, where the right hand side is recognized as the domain of symbols of $C^\infty$-differential operators on $\mathcal{M}$, and the left hand side is by the Poincar\'e-Birkhoff-Witt theorem.

In our situation, the isometric $G$-action on $N=G/K$ lifts to an isometric $G$-action on the unit tangent bundle $SN$. We have the following $G$-equivariant commutative diagram.
$$\begin{tikzcd}
\Diff^*_{\le2}(SN)\ar[r]&\Gamma(SN;\Sym(T(SN)))\ar[r]&\Gamma(SN;\Sym(T(SN)/T^V(SN)))\ar[sloped,dd,"\cong"]\\
U_{\le2}(\Lg)\ar[u]\ar[d]\ar[r]&\Sym\Lg\ar[u]\ar[d]\arrow[l,bend right=25]&\\
\Diff^*_{\le2}(N)\ar[r]&\Gamma(N;\Sym(TN))\ar[r,hook,"\pi^*"]&\Gamma(SN;\pi^*(\Sym(TN))).
\end{tikzcd}$$
Here, the maps from the first column to the second are symbol maps, $\Sym\Lg\subset\Lg\otimes\Lg\to U_{\le2}(\Lg)$ is a section of $U_{\le2}(\Lg)\to\Sym\Lg$, $\pi\colon SN\to N$ is the projection, and $T^V(SN)$ is the vertical subbundle of $T(SN)$.

Now pass to the $\Gamma$-quotient $M=\Gamma\backslash N$.  Taking $\Gamma$-invariant elements gives (still using $\pi\colon SM\to M$ for the projection)
$$\begin{tikzcd}
\Diff^*_{\le2}(SM)\ar[r]&\Gamma(SM;\Sym(T(SM)))\ar[r]&\Gamma(SM;\Sym(T(SM)/T^V(SM)))\ar[sloped,dd,"\cong"]\\
U_{\le2}(\Lg)^\Gamma\ar[u]\ar[d]\ar[r]&(\Sym\Lg)^\Gamma\ar[u]\ar[d]\arrow[l,bend right=25]&\\
\Diff^*_{\le2}(M)\ar[r]&\Gamma(M;\Sym(TM))\ar[r,hook,"\pi^*"]&\Gamma(SM;\pi^*(\Sym(TM))).
\end{tikzcd}$$

The operators $\Delta_V,\Delta,X$ on $SN$ are defined only in terms of metric properties of $N$, therefore they are $G$-invariant. It follows that every vector field in the image of $\Lg\to\mathfrak{X}(SN)$ commutes with them. Since $\Lg$ generates $U(\Lg)$, the same is true for $U(\Lg)\to\Diff^*(SN)$. Also, since the $G$-action on $SN$ is symmetric, the image of $\Lg\to\mathfrak{X}(SN)$ consists of Killing fields, which are anti-symmetric as operators on $SN$. Therefore every operator in the image of $\Sym\Lg\to\Diff_{\le2}^*(SN)$ is symmetric. Passing to the quotient then implies every operator in the image of $(\Sym\Lg)^\Gamma\to\Diff^*(SM)$ is symmetric and commutes with $\Delta_V,\Delta,X$ on $SM$. 

By assumption, there is an element $\omega\in(\Sym\Lg)^\Gamma$ whose image in each $\Sym(\Lg/Ad_g\Lk)$ is negative definite. We show its image $\Omega\in\Diff^*_{\le2}(SM)$ in the above diagram satisfies our requirements. Condition (2) and the symmetric property in condition (1) has just been justified. For the negative definite property in (1), note the horizontal part of the symbol of a second order differential operator on $SM$ sits in the upper-right corner of the diagram, thus it suffices to show the image of $\omega$ in $\Gamma(M;\Sym(TM))$ is pointwisely negative definite. We include $\omega$ back to $\Sym\Lg$ and prove the equivalent statement that its image in $\Gamma(N;\Sym(TN))$ is pointwisely negative definite.

There is a natural identification between the tangent bundle $TN=T(G/K)$ and the associated bundle $G\times_K(\Lg/\Lk)$ where $K$ acts on $\Lg/\Lk$ via the adjoint action, which on each fiber $T_{gK}N$ is given by $$T_gG/T_gK\xrightarrow{\cong}(gK)\times_K(\Lg/\Lk),\ [L_{g,*}X]\mapsto[(g,[X])],\ \text{for }X\in\Lg\cong T_eG.$$ Thus $\mathfrak{X}(N)=\Gamma(G;\Lg/\Lk)^K$. The negative associated vector field map is given by $\Phi\colon\Lg\to\Gamma(G;\Lg/\Lk)^K$, $\xi\mapsto(g\mapsto-[Ad_{g^{-1}}\xi])$. The map $\Sym\Lg\to\Gamma(N;\Sym(TN))$ is then identified with the second symmetric power of $\Phi$, namely the map $\Sym\Lg\to\Gamma(G;\Sym(\Lg/\Lk))^K$, $\alpha\mapsto(g\mapsto[Ad_{g^{-1}}\alpha])$. The assumption on $\omega$ exactly says its image under this map is fiberwisely negative definite.
\end{proof}
In Lie theory, the Casimir element of a Lie algebra $\Lg$ with respect to an $Ad$-invariant nondegenerate symmetric bilinear form is the image of the dual of this bilinear form in $U(\Lg)$. We see in the previous proof that our choice of $\Omega$ exactly arises from (the negative of) a Casimir element in $U(\Lg)$ when $M$ is locally symmetric without euclidean factors. This clarifies the name ``Casimir operator'' in those special cases. Finally, in the locally symmetric case, it is not hard to see if we break $M$ into irreducible pieces and use appropriate multiples of the Killing form in each piece in the construction, the operator $\Omega$ on $SM$ will have horizontal symbol agreeing with the pullback of the symbol of the usual Laplacian $\Delta_M$ on the base $M$. In the next section we will see this more explicitly in a simpler case. The next section will not be relevant to our main goal, though.

\subsection{Examples: Locally isotropic spaces and constant curvature spaces}\label{sec2.2}
Recall that an \textit{isotropic space} is a Riemannian manifold $N$ whose isometry group acts transitively on unit tangent vectors, and a \textit{locally isotropic space} is a Riemannian manifold $M$ that is locally isometric to an isotropic space. Thus in particular, (locally) isotropic spaces are (locally) symmetric. In fact, (locally) isotropic spaces exactly consist of (locally) euclidean spaces and rank one (locally) symmetric spaces. Isotropic spaces are completely classified: euclidean spaces $\R^n=E(n)/O(n)$; spheres $S^n=SO(n+1)/SO(n)$; real, complex, quaternionic projective spaces $\mathbb{P}^n_\R=SO(n+1)/O(n)$, $\mathbb{P}^n_\C=SU(n+1)/U(n)$, $\mathbb{P}^n_\mathbb{H}=Sp(n+1)/Sp(n)\cdot Sp(1)$; the Cayley projective plane $\mathbb{P}^2_\mathbb{O}=F_4/Spin(9)$; real, complex, quaternionic hyperbolic spaces $\mathbb{H}^n_\R=SO^+(n,1)/SO(n)$, $\mathbb{H}^n_\C=SU(n,1)/U(n)$, $\mathbb{H}^n_\mathbb{H}=Sp(n,1)/Sp(n)\cdot Sp(1)$; the Cayley hyperbolic plane $\mathbb{H}^2_\mathbb{O}=F_4^*/Spin(9)$. For an exposition, a good reference is ~\cite{wolf2011spaces}.\medskip

Let $N=G/K$ be a noneuclidean irreducible symmetric space where $G$ is the isometry group of $N$. Then $\Lg$ is semisimple whose Killing form is nondegenerate and is negative definite on $\Lk$. Let $\Lp$ be the orthogonal complement of $\Lk$, which is $Ad|_K$-invariant and satisfies $[\Lp,\Lp]\subset\Lk$. We have an identification $TN\cong G\times_K\Lp$ as in the proof of Proposition~\ref{prop:Cas}. Thus, a $G$-invariant Riemannian metric on $N$ is the same as an $Ad|_K$-invariant inner product on $\Lp$. In fact, let $\kappa$ be an appropriate scalar times the Killing form on $\Lg$, then $\kappa|_\Lp$ gives back the Riemannian metric on $N$.

Another way to recover the Riemannian metric on $N$ from $\kappa$ is the following. Left translating $\kappa$ on $G$ equips $G$ with a bi-invariant pseudo-Riemannian metric. Then the Riemannian metric on $N=G/K$ is characterized by requiring the quotient map $\pi_0\colon G\to N$ to be a pseudo-Riemannian submersion. Moreover, since the Lie exponential map on $G$ agrees with the pseudo-Riemannian exponential map, we know $\pi_0$ has totally geodesic fibers, i.e. every fiber of $\pi_0$ is a totally geodesic submanifold of $G$.

It is a classical result that the negative Casimir element $-c(\Lg)\in U(\Lg)$ with respect to $\kappa$ acts as the Laplacian $\Delta_G$ on the pseudo-Riemannian manifold $G$, via the map $U(\Lg)\to\Diff^*(G)$ defined as in the proof of Proposition~\ref{prop:Cas}. Since $\pi_0$ has totally geodesic fibers, classical results for (pseudo-)Riemannian submersion tell that $\Delta_G\pi_0^*=\pi_0^*\Delta_N$ (~\cite[(3.10)]{bergery1982laplacians}). Therefore, $-c(\Lg)$ acts on $N$ as the Laplacian $\Delta_N$.\medskip

Now assume furthermore $N$ is isotropic. Then we can write $SN=G/H$ where $H$ is the isotopy group of a unit tangent vector $[(e,X_0)]\in eK\times_KS(\Lp)\cong S_{eK}N$ ($S(\Lp)$ is the unit sphere in $\Lp$). Let $\Ls$ be the orthogonal complement of $\Lh$ in $\Lk$, then like above, $T(SN)\cong G\times_H(\Ls\oplus\Lp)$. Under this identification, the horizontal/vertical subbundles are
\begin{equation}\label{eq:hor/vert}
T^H(SN)\cong G\times_H\Lp,\ T^V(SN)\cong G\times_H\Ls.
\end{equation} (To see the claim on the horizontal part, one has to note the connection on the principal $K$-bundle $K\hookrightarrow G\to N$ given by the horizontal distribution $\Lp$ induces the Levi-Civita connection on $N$ because it is metric-compatible and torsion-free (~\cite[Proposition~I.13]{kowalski2007generalized}).)

Let $g_1$ be the Riemannian metric on $SN$ induced from the one on $N$, namely the restriction of the Sasaki metric on $TN$. Let $g_2$ be the (pseudo-)Riemannian metric on $SN$ given by $\kappa|_{\Ls\oplus\Lp}$, or alternatively the one such that the projection $G\to G/H=SN$ is a pseudo-Riemannian submersion. By the same arguments as before, in view of the construction in the proof of Proposition~\ref{prop:Cas} we have the following.
\begin{cor}\label{cor:Cas_SN}
The negative Casimir element $-c(\Lg)$ acts on $SN$ as the Laplacian with respect to the (pseudo-)Riemannian metric $g_2$. In particular, a Casimir operator $\Omega$ as in Proposition~\ref{prop:Cas} can be chosen to be the Laplacian on $(SN,g_2)$.\qed
\end{cor}

However, from the Riemannian point of view the more natural metric to use is $g_1$. Therefore we need to relate these two metrics.

In general $g_1,g_2$ do not agree, but with respect to $SN\to N$, they induce the same horizontal/vertical splitting (the one given by the Levi-Civita connection on $N$) and are equal in the horizontal part (both equal to the lift of the Riemannian metric on $N$). Therefore the difference of $g_1,g_2$ is only in the vertical part. Since they are both $G$-invariant, each of them is determined by an $Ad|_H$-invariant nondegenerate symmetric bilinear form on $\Ls$. We have noted that $g_2$ is given by $\kappa|_\Ls$, which is positive/negative definite if $N$ is of compact/noncompact type, respectively.

\begin{lem}
Under the identifications~\eqref{eq:hor/vert}, the tautological map $T^V(SN)\xrightarrow{\cong}X^\perp\subset T^H(SN)$ is given by the $Ad|_H$-equivariant map
\begin{equation}\label{eq:iota}
\iota:=-ad_{X_0}\colon\Ls\xrightarrow{\cong}X_0^\perp.
\end{equation}
Here $X^\perp$ is the orthogonal complement in $T^H(SN)$ of the line bundle spanned by the geodesic vector field $X$, and $X_0^\perp$ is the $\kappa$-orthogonal complement in $\Lp$ of the unit vector $X_0$.
\end{lem}
\begin{proof}
At a point $gH\in G/H\cong SN$, a vertical tangent vector $v$ given by the $\frac{d}{dt}|_{t=0}$ of a curve $gk(t)H$ in $gK/H$ for some curve $k(t)\in K$, $t\in(-\epsilon,\epsilon)$, with $k(0)=e$ and $k'(0)\in\Ls$. Via $T^V_{gH}(SN)\cong\Ls$ we can write $v=k'(0)$. On the other hand, write $gH=[(g,X_0)]=(gK,[L_{g,*}X_0])\in gK\times_K\Lp\cong T_{gK}N$, then the tautological map is given by $T_{gH}^VSN\xrightarrow{\cong}([L_{g,*}X_0])^\perp\subset T_{gK}N\cong gK\times_K\Lp\cong\Lp$, under which $v$ is mapped to $-ad_{X_0}k'(0)$, because it is the $\frac{d}{dt}|_{t=0}$ of $gk(t)H=[(gk(t),X_0)]=[(g,Ad_{k(t)}X_0)]=Ad_{k(t)}X_0\in\Lp$ under this sequence of identifications. The rest statements are straightforward.
\end{proof}
\begin{cor}\label{cor:g_1}
The vertical part of the Riemannian metric $g_1$ on $SN$ is given by $\iota^*\kappa|_{X_0^\perp}$ on $\kappa|_\Ls$ where $\iota$ is defined as in \eqref{eq:iota}.\qed
\end{cor}
We can also translate the Lie-theoretic Corollary~\ref{cor:g_1} entirely into the language of Riemannian geometry.

\begin{lem}\label{lem:curvature}
At a point $(x,v)\in SN$, the vertical part of $g_2$ (as a bilinear form on $T^V_{(x,v)}(SN)$) is given by the inverse of $R_x(\cdot,v,v,\cdot)$ with respect to the vertical part of $g_1$. Here $R$ denotes the Riemannian curvature tensor of $N$.
\end{lem}
\begin{proof}
Write $(x,v)=gH\in G/H\cong SN$, $g\in G$, and use this to make the identification $T_xN\cong gK\times_K\Lp\cong\Lp$. Then $v=X_0$. For $X_1,X_2\in X_0^\perp\cong v^\perp\cong T^V_{(x,v)}(SN)$, we have $$(g_2)_v(X_1,X_2)=(\iota^{-1})^*(\kappa|_\Ls)(X_1,X_2)=\kappa(ad_{X_0}^{-1}X_1,ad_{X_0}^{-1}X_2)=\kappa(-ad_{X_0}^{-2}X_1,X_2).$$ On the other hand, the Riemannian curvature tensor is given by $R_x(X,Y,Z,W)=-\kappa([X,Y],$ $[Z,W])$, for $X,Y,Z,W\in\Lp$ (~\cite[Proposition~I.13]{kowalski2007generalized}~\cite[Section 8.4]{wolf2011spaces}). Thus $$R_x(X_1,v,v,X_2)=-\kappa(-ad_{X_0}X_1,ad_{X_0}X_2)=\kappa(-ad_{X_0}^2X_1,X_2).$$ Finally, $$(g_1)_v(X_1,X_2)=\kappa(X_1,X_2).$$ Comparing the three expressions gives the desired statement.
\end{proof}

Note that the discussions above carry to the local case. From the Lie point of view, $g_1,g_2,\Omega$ are $G$-invariant, thus local constructions of $\Omega$ on a locally isotropic space glue to a global Casimir operator $\Omega$ (since transition functions take values in $G$). From the Riemannian point of view, the description in Proposition~\ref{lem:curvature} enables us to write down the expression of $\Omega$ in terms of local metric data on a locally isotropic space $M$. (Explicitly, if $x^i$'s are local coordinates on $M$ and $x^i,v^i$'s are induced coordinates on $TM$, then $\Omega$ is the restriction of $-g^{ij}(\bar\partial_i\bar\partial_j-\Gamma_{ij}^k\bar\partial_k)-R^i{}_{vv}{}^j\partial_{v^i}\partial_{v^j}$ to $SM$, where $g$ is the metric tensor on $M$, $\Gamma$ is the Christoffel symbol on $M$, and $\bar{\partial_i}=\partial_{x^i}-v^j\Gamma_{ij}^k\partial_{v^k}$ is the horizontal lift of $\partial_{x^i}$ on the base. It could be a tedious exercise to directly show from this expression that $\Omega$ commutes with the vertical Laplacian, total Laplacian, and the geodesic vector field, and try to see what fails when $M$ is not locally isotropic.)

\begin{cor}\label{cor:const_curv}
If $M$ is a Riemannian manifold with constant curvature $K$, then one can choose $\Omega=\Delta_H+K\Delta_V$. Here $\Delta_H:=\Delta-\Delta_V$ is called the horizontal Laplacian on $SM$.
\end{cor}
\begin{proof}
If $K=0$, the statement is clear from our construction in Corollary~\ref{cor:loc_sym}. If $K\ne0$, $\Omega=\Delta'$, where $(\cdot)'$ means operators on $SM$ with respect to the metric $g_2$ instead of $g_1$. Since $SM\to M$ is a Riemannian submersion with totally geodesic fibers for both $g_1$ and $g_2$ on $SM$, we know $\Delta_H'=\Delta_H$ (see e.g.~\cite[(1.9)]{bergery1982laplacians}). On the other hand, Proposition~\ref{lem:curvature} tells us the vertical part of $g_2$ is $K^{-1}$ times that of $g_1$, thus $\Delta_V'=K\Delta_V$. The statement follows.
\end{proof}

\section{Convergence of spectrum}\label{sec3}
In this section we prove the following slightly refined (c.f. Corollary~\ref{cor:loc_sym}) version of Theorem~\ref{thm:Conv}.
\begin{thm}\label{thm:conv_refined}
Suppose $M=\Gamma\backslash G/K$ is a locally Riemannian homogeneous space and assume there exists a $\Gamma$-invariant element in $\Sym\Lg$ whose image in $\Sym(\Lg/Ad_g\Lh)$ is define for every $g\in G$, then the conclusion of Theorem~\ref{thm:Conv} holds.
\end{thm}
The strategy is to decompose $L^2(SM)$ into eigenspaces of a Casimir operator constructed in Proposition \ref{prop:Cas} and then use the Schur complement formula on each eigenspace. For readers' convenience we give a brief review of the Schur complement formula.

\subsection{Schur Complement formula}
Schur's complement formula is a convenient tool to analyze the degeneracy of an operator.
\begin{lem}\label{lem:Schur}
Suppose 
\begin{align}\label{2.2.1}
\begin{pmatrix}
P&R_-\\
R_+&R_{+-}
\end{pmatrix}
=
\begin{pmatrix}
E&E_+\\
E_-&E_{-+}
\end{pmatrix}^{-1}:X_1\times X_-\to X_2\times X_+  
\end{align}
are bounded operators on Banach spaces, then $P$ is invertible if and only if $E_{-+}$ is invertible.
Moreover, in such case we have
\begin{align}\label{Grushin}
    P^{-1}=E-E_+E_{-+}^{-1}E_-,\quad E_{-+}^{-1}=R_{+-}-R_+P^{-1}R_-.
\end{align}
\end{lem}
\begin{proof}
The proof is direct. If $E_{-+}$ is invertible, then from
\begin{align*}
    PE+R_-E_-=I,\quad PE_++R_-E_{-+}=0,
\end{align*}
we get
$PE-PE_+E_{-+}^{-1}E_-=I$.
Similarly, since
\begin{align*}
    EP+E_+R_+=I,\quad E_-P+E_{-+}R_+=0,
\end{align*}
we get
$EP-E_+E_{-+}^{-1}E_-P=I$.
We conclude that $P$ is invertible and $P^{-1}=E-E_+E_{-+}^{-1}E_-$. The proof for the other direction is similar.
\end{proof} 
This lemma will be used for the $R_{+-}=0$ case, which we call a Grushin problem for $P$.
\subsection{Decomposition of \texorpdfstring{$L^2(SM)$}{L\^{}2(SM)}}\label{sec:decomp}
Fix a Casimir operator $\Omega$ on $SM$ constructed in Proposition~\ref{prop:Cas}, and as before let $\Delta_V$ be the vertical Laplacian on $SM$. Since $\Omega$, $\Delta_V$ and the total Laplacian $\Delta$ on $SM$ commute with each other, we can do spectral decomposition on each eigenspace of $\Delta$. Thus we get the following orthogonal decomposition:
\begin{align}\label{decomposition}
    L^2(SM)=\bigoplus\limits_\eta V_\eta =\bigoplus\limits_{\eta,k}V_{\eta,k}
\end{align}
where $V_\eta=\{u\in L^2(SM): \Omega u=\eta u\}$ is the $\eta$-eigenspace of $\Omega$ and $V_{\eta,k}=\{u\in V_\eta: \Delta_V u=\lambda_k u\}$ is the $k$-th eigenspace of $\Delta_V$ in $V_\eta$. Here
$\lambda_k=k(k+n-2)$,  $k=0,1,2,\cdots$ are the eigenvalues of the spherical Laplacian. For sufficiently large constant $C>0$, $\Omega+ C\Delta_V$ is elliptic thanks to Proposition \ref{prop:Cas}, so each $V_{\eta,k}$ is finite dimensional and consists of smooth functions.

The operator $\Delta_M:H^2(M)\to L^2(M)$ a priori only acts on $M$. But we will also let it act on $SM$ via $\Pi\Delta_M\Pi:H^2(SM)\to L^2(SM)$. We will also use the notation $(\Delta_M-\lambda)^{-1}$ to abbreviate $\Pi(\Delta_M-\lambda)^{-1}\Pi:L^2(SM)\to L^2(SM)$, as we have done in the statement of Theorem~\ref{thm:Conv}.

We will do analysis for $P_\gamma=-\gamma X+c_n\gamma^2 \Delta_V$ (where $c_n=1/n(n-1)$) on each eigenspace $V_\eta$. The advantage is that 
$P_\gamma=P_\gamma+(\Omega-\eta)/C:V_\eta\cap H^2 \to V_\eta$ is an elliptic operator. Let $\Pi_{\eta,k}: V_\eta\to V_{\eta,k}$ denote the orthogonal projection with the abbreviated notation $\Pi=\Pi_{\eta,0}:V_\eta\to V_{\eta,0}$ and $\Pi^\perp=\id-\Pi: V_\eta\to V_{\eta,>0}$. 

\begin{lem}
The geodesic vector field
$X$ satisfies the following properties.
\begin{itemize}
    \item $X$ is anti-self-adjoint with respect to the natural $L^2(SM)$ norm defined via the metric;
    \item $X$ sends $V_{\eta,k}$ into $V_{\eta,k+1}\oplus V_{\eta,k-1}$
with the convention that $V_{\eta,-1}=0$;
\item $n\Pi X^2\Pi=-\Delta_M$.
\end{itemize}
\end{lem}
\begin{proof}
\begin{itemize}
    \item We identify $SM$ with the unit cotangent bundle $S^*M$ via the metric on $M$. The geodesic vector field on $S^*M$ is the Reeb vector field with respect to the natural contact $1$-form $\alpha$. Therefore $X$ preserves both $\alpha$ and $d\alpha$, thus preserves the volume form on $S^*M$. It follows that $X$ is anti-self-adjoint.
    \item We compute in local coordinates. Suppose we are in normal coordinates $\{x^i\}$ at $p\in M$, so that $g_{ij}(p)=\delta_{ij}$ and $\partial_kg_{ij}(p)=0$. Then at $p$, $X=\sum v^j\partial_{x^j}$ where $v^j$'s are the induced coordinates on $TM$, which restrict to linear functionals on each fiber. The claim follows from the fact that multiplying spherical harmonics of degree $k$ by linear functionals gives a combination of spherical harmonics in degree $k-1$ and $k+1$.
    \item Again we compute in normal coordinates at $p\in M$. Then $X^2=\sum v^i v^j\partial_{x^i}\partial_{x^j}$. $\Pi X^2\Pi$ corresponds to projecting $v^iv^j$ to $0$-th spherical harmonics. If $i\neq j$, then the projection is zero. If $i=j$, then the projection is given by
    \begin{align*}
       \frac{1}{\Vol(S^{n-1})}\int_{S^{n-1}}(v^i)^2d\sigma(v)=\frac{1}{n\Vol(S^{n-1})}\int_{S^{n-1}}\sum\limits_{i=1}^n(v^i)^2d\sigma(v)=\frac{1}{n}.
    \end{align*}
    Thus $\displaystyle\Pi X^2\Pi=\frac{1}{n}\sum\limits_{i=1}^n\partial^2_{x^i}=-\frac{1}{n}\Delta_M.$\qedhere
\end{itemize}
\end{proof}

\subsection{Invertibility lemma}
We first prove an invertibility lemma which will be useful in our analysis later. We will use the semiclassical notation $h=\gamma^{-1}$ and $\tilde{P}_h=c_n\Delta_V-hX$. Note that $P_\gamma=\gamma^2\tilde{P}_h$.

\begin{lem}\label{lem:inv_c}
For $|\lambda|\leq C_0$, $2C_0h^2\leq  c_n$, the operator
\begin{align}
    Q_0=\Pi^\perp \left(\tilde{P}_h-h^2\lambda\right)\Pi^\perp: \Pi^\perp V_\eta\cap H^2\to \Pi ^\perp V_\eta.
\end{align}
is invertible with inverse $Q_0^{-1}$ satisfying
\begin{align*}
    \|Q_0^{-1}\|_{L^2\to L^2}\leq \frac{2}{c_n}.
\end{align*}

\end{lem}
\begin{proof}
First, for $u\in \Pi^\perp V_\eta \cap H^2$,
\begin{align*}
    \|Q_0 u\|_{L^2}\|u\|_{L^2}\geq \Re((c_n\Delta_V-hX-h^2\lambda )u,u)=(\left(c_n\Delta_V-h^2\Re \lambda\right)u,u)\geq \frac{c_n}{2}\|u\|_{L^2}^2,
\end{align*}
and hence we know $Q_0$ is injective, and the image is closed by ellipticity. 

Now suppose $Q_0$ is not surjective, then there exists a nonzero $v\in \Pi^\perp V_\eta$ such that
\begin{align*}
    (Q_0 u,v)=0,\text{ for any }u\in \Pi^\perp V_\eta\cap C^\infty.
\end{align*}
Then
$\left(u,\left(c_n\Delta_V+hX-h^2\bar{\lambda}\right)v\right)=0$ for all $u\in \Pi^\perp V_\eta\cap C^\infty$
and thus
$\Pi^\perp\left(c_n\Delta_V+hX-h^2\bar{\lambda}\right)v=0$
as a distribution. Notice
\begin{align*}
    (c_n\Delta_V+hX-h^2\bar{\lambda})v\in V_{\eta,0}\subset C^\infty.
\end{align*}
By ellipticity, $v\in C^\infty$ and this is a contradiction to injectivity.
\end{proof}

\subsection{Spectral convergence}
In this section we prove the convergence of the spectrum in Theorem~\ref{thm:conv_refined}.

Let $i_0: V_{\eta,0}\to V_\eta$ be the inclusion. We will consider the following Grushin problem for $P_\gamma-\lambda$.
\begin{align*}
    \begin{pmatrix}
    P_\gamma-\lambda &  \gamma i_0\\
    \gamma\Pi&0
    \end{pmatrix}:
    V_\eta\cap H^2 \oplus  V_{\eta,0}\to V_\eta\oplus V_{\eta,0}.
\end{align*}
We would like to use Lemma \ref{lem:Schur}, so let us check the operator is invertible. In other words, we want to solve the equations
\begin{align}\label{equ: grushin}
    \left\{\begin{array}{ll}
    (P_\gamma-\lambda) u+\gamma u_- =v,    &  \\
    \gamma\Pi u=v_+.     & 
    \end{array}\right.
\end{align}
Thanks to Lemma \ref{lem:inv_c}, we can explicitly solve the equation. From the second equation of \eqref{equ: grushin} one knows the $\Pi$ part of $u$. The $\Pi^\perp$ part of the first equation of \eqref{equ: grushin} then gives the $\Pi^\perp$ part of $u$. Plugging the solution $u$ back gives $u_-$. Hence the unique solution of \eqref{equ: grushin} is
\begin{align}\label{sol:grushin}
    \left\{\begin{array}{cl}
    u&=(\Pi^\perp (P_\gamma-\lambda)\Pi^\perp)^{-1}\Pi^\perp (v+X v_+)+\gamma^{-1} v_+,    \\
    u_-&= \gamma^{-1}\Pi v+\gamma^{-2}\lambda v_++\Pi X(\Pi^\perp (P_\gamma-\lambda)\Pi^\perp)^{-1}\Pi^\perp (v+Xv_+).
    \end{array}\right.
\end{align}
Now we can apply Lemma \ref{lem:Schur} and it follows that $\lambda\in \sigma(P_\gamma|_{V_\eta})$ if and only if $E_{-+}=\gamma^{-2} (\lambda +\Pi X(\Pi^\perp (\tilde{P}_h-h^2\lambda)\Pi^\perp)^{-1}X\Pi)$ is not invertible.
Let us analyze the asymptotic behavior of $E_{-+}$.
\begin{prop}
Fix $\eta\in \sigma(\Omega)$, then on the finite-dimensional space $V_{\eta,0}$
\begin{align*}
    \lambda +\Pi X(\Pi^\perp (\tilde{P}_h-h^2\lambda)\Pi^\perp)^{-1}X\Pi \to \lambda-\Delta_M
\end{align*}
locally uniformly in $\lambda$ on $V_{\eta,0}$, as $h\to 0$.
\end{prop}
\begin{proof}
We have the estimate for the difference
\begin{align*}
    &\Pi_{\eta,1}(\Pi^\perp (\tilde{P}_h-h^2\lambda)\Pi^\perp)^{-1}\Pi_{\eta,1}-\Pi_{\eta,1}(\Pi^\perp (c_n\Delta_V)\Pi^\perp)^{-1}\Pi_{\eta,1}\\
    =\,&\Pi_{\eta,1}(\Pi^\perp (\tilde{P}_h-h^2\lambda)\Pi^\perp)^{-1}(h\Pi_{\eta,2}X+h^2\lambda)(\Pi^\perp (c_n\Delta_V)\Pi^\perp)^{-1}\Pi_{\eta,1}\\
    =\,& n\Pi_{\eta,1}(\Pi^\perp (\tilde{P}_h-h^2\lambda)\Pi^\perp)^{-1}(h\Pi_{\eta,2}X+h^2\lambda)\Pi_{\eta,1}.
\end{align*}
Since
$(\Pi^\perp (\tilde{P}_h-h^2\lambda)\Pi^\perp)^{-1},\Pi_{\eta,2}X, \Pi_{\eta,1}, \lambda$
are all uniformly bounded for $\lambda$ in a compact set, we see as $h\to 0$,
\begin{align*}
    \Pi_{\eta,1}(\Pi^\perp (\tilde{P}_h-h^2\lambda)\Pi^\perp)^{-1}\Pi_{\eta,1}\to \Pi_{\eta,1}(\Pi^\perp (c_n\Delta_V)\Pi^\perp)^{-1}\Pi_{\eta,1}=n\Pi_{\eta,1}
\end{align*}
locally uniformly. Note $V_{\eta,0}$ is finite dimensional, we have
\[
    \lambda +\Pi X(\Pi^\perp (\tilde{P}_h-h^2\lambda)\Pi^\perp)^{-1}X\Pi\to \lambda +n\Pi X^2\Pi=\lambda-\Delta_M.\qedhere
\]
\end{proof}
As a corollary, we get
\begin{cor}\label{cor:spec_conv}
For each $\eta$,
$\sigma(P_\gamma|_{V_\eta})\to \sigma(\Delta_M|_{V_{\eta,0}})$ locally uniformly
as $\gamma\to \infty$.\qed
\end{cor}

Note $L^2(SM)=\bigoplus V_\eta$ and $L^2(M)=\bigoplus V_{\eta,0}$, we would like to conclude that $\sigma(P_\gamma)\to \sigma(\Delta_M)$. However, it is unclear a priori whether the convergence is uniform in $\eta$.
We now prove the stronger convergence in Theorem~\ref{thm:conv_refined} using contradiction. It suffices to prove the following proposition.
\begin{prop}\label{prop: contra}
For any $C_0>0$, there exists $C_1=C_1(C_0)>0$ such that for any $\gamma>C_1$ and $|\eta|>C_1$ we have
\begin{align*}
    \sigma(P_\gamma|_{V_\eta})\cap \{|\lambda|\leq C_0\}=\varnothing.
\end{align*}
\end{prop}
\begin{proof}
We proceed by contradiction. Suppose there exists $u\in V_\eta$ with $\|u\|_{L^2}=1$ such that for some $|\lambda|\leq C_0$,
\begin{align}\label{eigenequ}
    P_\gamma u = \lambda u.
\end{align}
Then
$c_n\gamma^2(\Delta_V u,u)=\Re(P_\gamma u,u) =\Re(\lambda)\leq C_0$
and thus
\begin{align}\label{3.6}
    \|\Pi_{\eta,>0} u\|_{L^2}^2\leq c_n^{-1} C_0\gamma^{-2},\quad \|\Pi_{\eta,0}u\|_{L^2}= 1-\Ocal(\gamma^{-2}).
\end{align}
Let us compare the $\Pi_{\eta,1}$ component of \eqref{eigenequ}.
\begin{align}\label{pi1eigen}
    \frac{\gamma^2}{n}\Pi_{\eta,1}u-\gamma \Pi_{\eta,1}Xu=\lambda\Pi_{\eta,1}u.
\end{align}
By \eqref{3.6}, $\|\Pi_{\eta,1}u\|_{L^2}\leq c_n^{-1/2}C_0^{1/2}\gamma^{-1}$ for the first term on the left. The second term decomposes as $\Pi_{\eta,1}Xu=\Pi_{\eta,1}X\Pi_{\eta,0}u+\Pi_{\eta,1}X\Pi_{\eta,2}u$. We make use of the operator $\Pi\Omega\Pi$ to estimate each of the two terms. By the construction in Proposition \ref{prop:Cas}, $\Pi\Omega\Pi$ is a nonnegative self-adjoint second order elliptic operator on $M$ which commutes with $\Delta_M$. 
In particular, $\eta$ is nonnegative whenever $V_{\eta,0}\neq 0$. Therefore we have the following estimates.
\begin{align*}
    &\|\Pi_{\eta,1}X\Pi_{\eta,0}u\|_{L^2}^2=-(\Pi_{\eta,0}X^2\Pi_{\eta,0}u,u)=\frac{1}{n}(\Delta_M \Pi_{\eta,0}u,\Pi_{\eta,0}u)\\
    &=\frac{1}{n}(\Delta_M(\Pi\Omega\Pi)^{-1}\Omega \Pi_{\eta,0}u,\Pi_{\eta,0}u)
    \gtrsim |\eta|\|\Pi_{\eta,0}u\|_{L^2}^2\gtrsim |\eta|.
\end{align*}
\begin{align*}
    \|\Pi_{\eta,1}X\Pi_{\eta,2}u\|_{L^2}^2\lesssim \|\Pi_{\eta,2}u\|_{H^1}^2\lesssim ((\Omega+C) \Pi_{\eta,2}u,\Pi_{\eta,2}u)\lesssim (|\eta|+1)\|\Pi_{\eta,2}u\|_{L^2}^2\lesssim \gamma^{-1}(|\eta|+1).
\end{align*}
For $\gamma$ large enough, we get $|\eta|\lesssim 1$ from \eqref{pi1eigen}.
Taking $C_1$ large enough gives a contradiction.
\end{proof}
Now we can finish the proof of spectrum convergence \eqref{spec_conv}. Fix a region $U\Subset\CC$, by Proposition \ref{prop: contra}, for $\gamma\to \infty$, we only need to consider $|\eta|\leq C_1$ for some constant $C_1$. If $V_{\eta,0}=0$, then $V_\eta=\Pi^\perp V_\eta$ and by Lemma \ref{lem:inv_c} $P_\gamma|_{V_\eta}$ has no eigenvalue in $U$. We are left with the case $V_{\eta,0}\neq 0$. Note the set $\{\eta:|\eta|\leq C_1,V_{\eta,0}\neq 0\}$ is finite as $\Pi\Omega\Pi$ is elliptic. Then \eqref{spec_conv} follows from Corollary \ref{cor:spec_conv}.

\subsection{The resolvant convergence}
We now prove the resolvant convergence \eqref{res_conv} in the Theorem~\ref{thm:conv_refined}. Recall from the Grushin problem \eqref{Grushin} we have
\begin{align}
    (P_\gamma-\lambda)^{-1}&=E-E_+E_{-+}^{-1}E_-
\end{align}
where by \eqref{sol:grushin},
\begin{align*}
    E&=(\Pi^\perp (P_\gamma-\lambda)\Pi^\perp)^{-1}\Pi^\perp,
    &E_+&=(\Pi^\perp (P_\gamma-\lambda)\Pi^\perp)^{-1}\Pi^\perp X+\gamma^{-1},\\
   E_-&=\gamma^{-1}\Pi+\Pi X(\Pi^\perp (P_\gamma-\lambda)\Pi^\perp)^{-1}\Pi^\perp,
   &E_{-+}&=\gamma^{-2}( \lambda +\Pi X(\Pi^\perp (\tilde{P}_h-h^2\lambda)\Pi^\perp)^{-1}X\Pi).
\end{align*}
On each $V_\eta$, by Lemma \ref{lem:inv_c},
\begin{align*}
    \|E\|_{L^2\to L^2}\lesssim \gamma^{-2}, \|E_+-\gamma^{-1}\|_{L^2\to L^2}\lesssim_\eta \gamma^{-2}, \|E_--\gamma^{-1}\Pi\|_{L^2\to L^2}\lesssim_\eta \gamma^{-2}.
\end{align*}
This implies on each $V_\eta$,
\begin{align}\label{res_conv_eta}
    (P_\gamma-\lambda)^{-1}\to (\Delta_M-\lambda)^{-1},\quad \gamma\to \infty.
\end{align}
To get a uniform estimate for all $\eta$, it suffices to establish the following proposition as before. 
\begin{prop}\label{prop:res}
For any $C_0>0$, $\epsilon>0$, there exists $C_1=C_1(C_0,\epsilon)>0$ such that for any $|\lambda|\leq C_0$, $|\eta|>C_1$, $\gamma>C_1$ we have on $V_\eta$
\begin{align*}
    \|(P_\gamma-\lambda)^{-1}\|_{L^2\to L^2}\leq \epsilon,\quad \|(\Delta_M-\lambda)^{-1}\|_{L^2\to L^2}\leq \epsilon.
\end{align*}
\end{prop}
\begin{proof}
If $C_1\gg \epsilon^{-1}+C_0$ then we have on $V_{\eta,0}$
\begin{align*}
    \|(\Delta_M-\lambda)^{-1}\|_{L^2\to L^2}\leq \|(\Pi\Omega\Pi/C-\Re\lambda)^{-1}\|_{L^2\to L^2}\leq (|\eta|/C-|\lambda|)^{-1}\leq \epsilon.
\end{align*}
On the other hand, if 
$\|(P_\gamma-\lambda)^{-1}\|_{L^2\to L^2}\leq \epsilon$
does not hold on some $V_\eta$ with $|\eta|>C_1$, then there exists $u\in V_\eta\cap H^2$ with
\begin{align}\label{contraequ}
    P_\gamma u=\lambda u +v,\quad \|u\|_{L^2}=1, \|v\|_{L^2}<\epsilon^{-1}.
\end{align}
Pairing with $u$ and taking the real part, we get
$c_n\gamma^2(\Delta_Vu,u)=\Re(\lambda)+\Re(v,u)\leq C_0+\epsilon^{-1}$. Therefore, 
\begin{align*}
    \|\Pi_{\eta,>0} u\|_{L^2}^2\leq c_n^{-1}(C_0+\epsilon)\gamma^{-2},\quad \|\Pi_{\eta,0}u\|_{L^2}=1-\Ocal(\gamma^{-2}).
\end{align*}
As before we compare the $\Pi_{\eta,1}$ component of \eqref{contraequ}.
\begin{align}\label{pi1contra}
    \frac{\gamma^2}{n}\Pi_{\eta,1}u-\gamma\Pi_{\eta,1}Xu=\lambda\Pi_{\eta,1}u+\Pi_{\eta,1}v.
\end{align}
Similar to the proof of Proposition \ref{prop: contra}, we have $\|\Pi_{\eta,1}u\|_{L^2}\lesssim \gamma^{-1}$, $\|\Pi_{\eta,1}X\Pi_{\eta,0}u\|_{L^2}^2\gtrsim |\eta|$ and
$\|\Pi_{\eta,1}X\Pi_{\eta,2}u\|_{L^2}^2\lesssim \gamma^{-1}(|\eta|+1)$.
Moreover, $\|\Pi_{\eta,1}v\|_{L^2}\leq \|v\|_{L^2}< \epsilon^{-1}$. For $\gamma$ large enough, putting all these estimates into \eqref{pi1contra} we will get $|\eta|\lesssim_{\epsilon,C_0} 1$. Choosing a large enough $C_1$ gives a contradiction.
\end{proof}
Finally we can finish the proof of the resolvent convergence in \eqref{res_conv}. Fix a region $U\Subset\CC\setminus\sigma(\Delta_M)$. We need to prove for any $\epsilon>0$, there is $\gamma_0>0$ such that for all $\gamma>\gamma_0$,
\begin{align*}
    \|(P_\gamma-\lambda)^{-1}-(\Delta_M-\lambda)^{-1}\|_{L^2\to L^2}\leq \epsilon.
\end{align*}

Thanks to Proposition \ref{prop:res}, it suffices to consider $|\eta|\leq C_1$ for some constant $C_1=C_1(U,\epsilon)>0$. If $V_{\eta,0}=0$, then $V_\eta=\Pi^\perp V_\eta$ and by Lemma \ref{lem:inv_c}, $\|(P_\gamma-\lambda)^{-1}\|_{L^2\to L^2}\lesssim \gamma^{-2}$. We are left with the case $V_{\eta,0}\neq 0$ as before. Since the set $\{\eta:|\eta|\leq C_1, V_{\eta,0}\neq 0\}$ is finite, the claim follows from \eqref{res_conv_eta}.







\printbibliography








\end{document}